\documentclass[reqno,12pt]{amsart}
\usepackage{epsfig,amscd,amssymb,amsmath,amsfonts}

\usepackage{amsmath}
\usepackage{graphicx}
\usepackage{lscape}
\usepackage{amsthm,color}
\usepackage{tikz}
\usepackage{multirow}
\usetikzlibrary{graphs}
\usetikzlibrary{graphs,quotes}
\usetikzlibrary{decorations.pathmorphing}

\tikzset{snake it/.style={decorate, decoration=snake}}
\tikzset{snake it/.style={decorate, decoration=snake}}

\usetikzlibrary{decorations.pathreplacing,decorations.markings,snakes}

\newtheorem{theorem}{Theorem}[section]

\theoremstyle{definition}
\newtheorem{definition}[theorem]{Definition}

\newtheorem{example}[theorem]{Example}

\theoremstyle{remark}
\newtheorem{remark}[theorem]{Remark}

\numberwithin{equation}{section}

\usepackage[margin=1.05in]{geometry}
\usepackage[colorlinks]{hyperref}
\usepackage{siunitx}

\makeatletter
\let\@wraptoccontribs\wraptoccontribs
\makeatother
\begin{document}
 \title{The $r$-equilibrium Problem}

\author{Mihai D. Staic}
\address{Department of Mathematics and Statistics, Bowling Green State University, Bowling Green, OH 43403, United States. }
\address{Institute of Mathematics of the Romanian Academy, PO.BOX 1-764, RO-70700 Bu\-cha\-rest, Romania.}

\email{mstaic@bgsu.edu}

 \begin{abstract}
In this paper we introduce the $r$-equilibrium problem and discuss connections to the map $det^{S^r}$. The case $r=2$ is an application of Newton's third law of motion, while $r=3$ deals with equilibrium of torque-like forces. 
 \end{abstract}


 \maketitle

\section{Introduction}

The notion of linear dependence and the determinant map play an important role virtually in all applications of linear algebra. One such example is finding equilibrium conditions in Newtonian mechanics. More precisely, suppose that we have one particle in the two dimensional space $\mathbb{R}^2$, and there are $3$ external forces $F_1$, $F_2$, and $F_3$ that act on that particle. 
Then it is always possible to find $\lambda_1$, $\lambda_2$, $\lambda_3$ not all zero such that 
$$\lambda_1F_1+\lambda_2F_2+\lambda_3F_3=0,$$
in other words, the forces  can be rescaled to equilibrium. 

However, if we have only two external forces $F_1$ and $F_2$ that act on the particle, then we can rescale them to equilibrium if and only if  the determinant of the corresponding $2\times 2$ matrix is zero. For example, if 
$F_1=\begin{pmatrix}
1\\
0
\end{pmatrix}$ and 
$F_2=\begin{pmatrix}
0\\
1
\end{pmatrix}$ then we cannot find $\lambda_1$ and $\lambda_2$ (not both equal to zero) such that $\lambda_1F_1+\lambda_2F_2=0$.

Obviously, the remark above is a particular case of the  well know linear algebra fact that any collection of $d+1$ vectors in a $d$-dimensional vector space $V_d$ is linearly dependent. Moreover, a collection of $d$ vectors in $V_d$ will be linearly dependent if and only if the determinant of the corresponding matrix is zero. In this paper the above setting will be called the $1$-equilibrium problem.

The maps $det^{S^r}$ were introduced in  \cite{ls2,dets2}, and they are a byproduct of the exterior algebra-like construction $\Lambda_{V_d}^{S^r}$ from \cite{sta2,S3}.  When $r=1$ one can identify $\Lambda_{V_d}^{S^1}$ and $det^{S^1}$ with the exterior algebra $\Lambda_{V_d}$ and the determinant map $det$ respectively. 
It was conjectured in \cite{sta2,S3} that $dim_k(\Lambda_{V_d}^{S^r}[rd])=1$. 
This conjecture is equivalent with the existence and uniqueness of a nontrivial linear map $det^{S^r}:V_d^{\otimes\binom{rd}{r}}\to k,$ with a certain universality property that will be recalled in Section \ref{section1}. The existence of such a map $det^{S^r}$ is known \cite{ls2,dets2}, although for $r\geq 4$ it's not known yet if that map is nontrivial for every $d$. The uniqueness is an open question in all but a few cases. It was shown in \cite{ls2,dets2} that the map $det^{S^r}$ can be used to detect if a $d$-partition of the complete $r$-uniform hypergraph $K_{rd}^r$ has zero Betti numbers or not. Also, one can show that the maps $det^{S^r}$ are invariant under the natural action of the special linear group $SL_d(k)$.

In this paper we introduce the $r$-equilibrium problem and discuss its connection to the map $det^{S^r}$. To make the presentation self contained, in Section \ref{section1} we recall the definition and some properties of the map $det^{S^r}$. 

In Section \ref{section2} we discuss how Newton's third law of motion leads to the formulation of the $2$-equilibrium problem, and its relation to the map $det^{S^2}$. More precisely, we show that if in a $d$-dimensional space $V_d$ we have a system consisting of $2d$ particles that interact through forces $F_{i,j}$ which satisfy the third law of motion, then the system can be rescaled to equilibrium if and only if  $det^{S^2}(\otimes_{1\leq i<j\leq 2d}(F_{i,j}))=0$. Moreover, any system with a least $2d+1$ particle that interact through forces $F_{i,j}$ (which satisfy the third law of motion) can be rescaled non-trivially  to equilibrium.  

In Section \ref{section3} we consider a natural generalization with regard to torque-like forces and show how that relates to the map $det^{S^3}$. In order to set-up the problem we postulate a rule similar to the third law of motion and  investigate its consequences.  We present an example based on the cross product in $\mathbb{R}^3$, and discuss some related problems. Finally, in Section \ref{sectionr} we formulate the general $r$-equilibrium problem and briefly discuss its connection to the map $det^{S^r}$.

One should notice that most of the mathematics from this paper is not new, and it can be found in \cite{ls2,dets2}. What is new is the interpretation of maps $det^{S^r}$ in the context of $r$-equilibrium problems. It is worth noticing that while the map $det^{S^2}$ was introduced from purely algebraic and combinatorial  considerations, it has an explicit physics interpretation based on Newton's third law of motion. It is therefore somehow surprising that such a map was not studied before. It would be interesting to investigate if the other $r$-equilibrium problems discussed in this paper have corresponding real world physics laws.


\section{Preliminaries}
\label{section1}

In this paper $\mathbb{R}$ is the field of real numbers, $k$ is a field, $\otimes=\otimes_k$, $V_d$ is a vector space of dimension $d$, and $\{e_1,e_2,\dots, e_d\}$ is a basis for $V_d$. To make the presentation self contained, we recall some results about the maps $det^{S^r}$. 

It is well known that the determinant is the unique (up to a constant) nontrivial linear map $$det:V_d^{\otimes d}\to k,$$ with the property that $det(\otimes_{1\leq i\leq d}(v_{i}))=0$ if there exist $1\leq x< y\leq d$ such that $v_x=v_y$. 

The map $det^{S^2}$ was introduced and studied in  \cite{edge,sta2,dets2}. It is a nontrivial linear map $$det^{S^2}:V_d^{\otimes\binom{2d}{2}}\to k,$$ with the property that $det^{S^2}(\otimes_{1\leq i<j\leq 2d}(v_{i,j}))=0$ if there exist $1\leq x<y<z\leq 2d$ such that $v_{x,y}=v_{x,z}=v_{y,z}$.  

Since we will use it in this paper, we recall from \cite{dets2} the construction of $det^{S^2}$. Let $v_{i,j}\in V_d$ for $1\leq i<j\leq 2d$. For each $1\leq m\leq 2d$ define a vector equation $\mathcal{E}_m$ given by
\begin{equation}\label{dets2Rel}
    \sum_{s=1}^{m-1}(-1)^{s-1}\lambda_{s,m}v_{s,m}+\sum_{t=m+1}^{2d}(-1)^t\lambda_{m,t}v_{m,t}=0.
\end{equation}
One can show that the equations $\mathcal{E}_s$ for $1\leq s\leq 2d$ are not independent, but they satisfy the relation
\begin{equation}\label{dets2Eq}
    \sum_{s=1}^{2d}(-1)^s\mathcal{E}_s=0.
\end{equation}
In particular this means that the equation $\mathcal{E}_{2d}$ is a consequence of the first $2d-1$ equations $\mathcal{E}_l$ where $1\leq l\leq 2d-1$. And so, when we study the system consisting of equations  $\mathcal{E}_{l}$ for $1\leq l\leq 2d$, we can ignore the equation $\mathcal{E}_{2d}$. 

Notice that this new system  has $2d-1$ vector equations (which correspond to $d(2d-1)$ linear equations), while the number of variables is $ \displaystyle{\binom{2d}{2}=d(2d-1)}$. In particular, its corresponding matrix is a square matrix. The map $det^{S^2}$ is defined as the determinant of the matrix associated to the system consisting of the vector equations $\mathcal{E}_{l}$ for $1\leq l\leq 2d-1$.   

It was shown in \cite{dets2} that $det^{S^2}(\otimes_{1\leq i<j\leq 2d}(v_{i,j}))=0$ if there is $1\leq x<y<z\leq 2d$ such that $v_{x,y}=v_{x,z}=v_{y,z}$. Moreover, for each $d\geq 1$ there exists $E_d^{(2)}\in V_d^{\otimes\binom{2d}{2}}$ such that $det^{S^2}(E_d^{(2)})\neq 0$, and so the map is nontrivial. It is known that a map with the above properties is unique up to a constant for $d=2$ (\cite{sta2}), $d=3$  (\cite{edge}), and $d=4$  (\cite{fls}). For $d\geq 5$ the uniqueness of the map $det^{S^2}$ is still an open question. 
Some geometrical properties of the map $det^{S^2}$ were studied in \cite{sv}.

The map $det^{S^3}$ was introduced in  \cite{ls2,S3}. It is a nontrivial linear map $$det^{S^3}:V_d^{\otimes\binom{3d}{3}}\to k,$$ with  the property that $det^{S^3}(\otimes_{1\leq i<j<k\leq 3d}(v_{i,j,k}))=0$ if there exists $1\leq x<y<z<t\leq 3d$ such that $v_{x,y,z}=v_{x,y,t}=v_{x,z,t}=v_{y,z,t}.$ It was shown in \cite{S3} that for $d=2$ a map with the above property is unique up to a constant. For $d\geq 3$ the uniqueness of the map $det^{S^3}$ is still an open question.  

To define the map $det^{S^3}$ we consider vector equations $\mathcal{E}_{m,n}$  for all $1\leq m<n\leq 3d$ defined by 
\begin{equation}\label{MainEq3}
 \sum_{r=1}^{m-1}(-1)^{r-1}\lambda_{r,m,n}v_{r,m,n}+\sum_{s=m+1}^{n-1}(-1)^s\lambda_{m,s,n}v_{m,s,n}+\sum_{t=n+1}^{3d}(-1)^{t+1}\lambda_{m,n,t}v_{m,n,t}=0.
\end{equation}
One can show that for each $1\leq n\leq 3d$  we have a dependence relation $\mathcal{R}_n$ given by
 \begin{equation}\label{relations}
     \mathcal{R}_n: \quad\quad\quad \sum_{m=1}^{n-1}(-1)^m\mathcal{E}_{m,n}+\sum_{p=n+1}^{3d}(-1)^{p+1}\mathcal{E}_{n,p}=0.
 \end{equation}
In particular, it follows that when studying the system determined by the vector equations $\mathcal{E}_{m,n}$ for $1\leq m<n\leq 3d$ we can ignore the equations $\mathcal{E}_{m,3d}$ for all $1\leq m\leq 3d-1$. 

Notice that since $\displaystyle{\binom{3d}{3}=d\binom{3d-1}{2}}$, the number of variables $\lambda_{i,j,k}$ is equal to the number of linear equations, and so the matrix associated to the system consisting of the vector equations $\mathcal{E}_{m,n}$ for $1\leq m<n\leq 3d-1$ is a square matrix. We denote its determinant by $det^{S^3}(\otimes_{1\leq i<j<k\leq 3d}(v_{i,j,k}))$. 

One can check that	 $det^{S^3}(\otimes_{1\leq i<j<k\leq 3d}(v_{i,j,k}))=0$ if there exists $1\leq x<y<z<t\leq 3d$ such that $v_{x,y,z}=v_{x,y,t}=v_{x,z,t}=v_{y,z,t}$. Moreover, it was shown in \cite{ls2} that for every $d\geq 1$ there exists $E_d^{(3)}\in V_d^{\otimes\binom{3d}{3}}$ such that $det^{S^3}(E_d^{(3)})\neq 0$, and so the map in nontrivial. Finally, it was shown in \cite{ls2} that the map $det^{S^3}$ can be used to detect if a $d$-partition of the complete $3$-uniform hypergraph $K_{3d}^3$ has zero Betti numbers or not.

The map $det^{S^r}$ was introduced in  \cite{ls2,S3}. It is a linear map $$det^{S^r}:V_d^{\otimes\binom{rd}{r}}\to k,$$ with  the property that $det^{S^r}(\otimes_{1\leq i_1<i_2<\ldots<i_r\leq rd}(v_{i_1,\ldots,i_r}))=0$ if there exist $1\leq x_1<x_2<\ldots<x_{r+1}\leq rd$ such that
 $$v_{x_1,x_2,\ldots,x_r}=v_{x_1,\ldots,x_{r-1},x_{r+1}}=v_{x_1,\ldots,x_{r-2},x_r,x_{r+1}}=\ldots=v_{x_1,x_3,\ldots,x_r,x_{r+1}}=v_{x_2,x_3,\ldots,x_{r+1}}.$$
If $r\geq 4$ it is not known if the map $det^{S^r}$ is nontrivial or not (some particular cases were discussed in the appendix of \cite{ls2}). It was conjectured in \cite{sta2,S3} that a nontrivial map with the above property exists and it is unique up to a constant. The map $det^{S^r}$ can be used to detect if a $d$-partition of the complete $r$-uniform hypergraph $K_{rd}^r$ has zero Betti numbers or not. For more details, including the construction of  $det^{S^r}$, we refer to \cite{ls2}.


\section{The $2$-equilibrium Problem}

\label{section2} 
In this section we introduce the $2$-equilibrium problem and discuss its relation to the the map $det^{S^2}$. 
This consideration is based on Newton's third law of motion: {\it the forces that two particles exert on each other are equal and opposite}. 

As a warm-up, we consider the dimension two case. Take four particles in a two dimensional space that interact through forces $F_{i,j}\in \mathbb{R}^2$ for $1\leq i,j\leq 4$ (with the convention  that $F_{i,j}$ acts on the particle $i$, and $F_{j,i}$ acts on the particle $j$). 
By Newton's third law of motion we  must have $F_{i,j}=-F_{j,i}$ (in particular $F_{i,i}=0$). 
We would like to see under what conditions we can rescale the forces $F_{i,j}$ such that the system is in equilibrium. More precisely, we want to find scalars $\lambda_{i,j}=\lambda_{j,i}\in \mathbb{R}$ for all $1\leq i<j\leq 4$, not all zero that are a solution to the system
\[\begin{cases}
\lambda_{1,2}F_{1,2}+\lambda_{1,3}F_{1,3}+\lambda_{1,4}F_{1,4}=0\\
\lambda_{2,1}F_{2,1}+\lambda_{2,3}F_{2,3}+\lambda_{2,4}F_{2,4}=0\\
\lambda_{3,2}F_{3,1}+\lambda_{3,2}F_{3,2}+\lambda_{3,4}F_{3,4}=0\\
\lambda_{4,1}F_{4,1}+\lambda_{4,2}F_{4,2}+\lambda_{4,3}F_{4,3}=0.\\
\end{cases}\]
Notice that the first equation is the sum of all (rescaled) forces that act on the first particle, etc. So, the above system ensures that after rescaling, the total force which acts on the $i$-th particle is zero for all $1\leq i\leq 4$. 

Next, if  we denote $v_{i,j}=(-1)^{i+j-1}F_{i,j}$ for all $1\leq i<j\leq 4$, and we use the fact that $\lambda_{i,j}=\lambda_{j,i}$ and $F_{i,j}=-F_{j,i}$ to get
\[\begin{cases}
\lambda_{1,2}v_{1,2}-\lambda_{1,3}v_{1,3}+\lambda_{1,4}v_{1,4}=0\\
-\lambda_{1,2}v_{1,2}+\lambda_{2,3}v_{2,3}-\lambda_{2,4}v_{2,4}=0\\
\lambda_{1,3}v_{1,3}-\lambda_{2,3}v_{2,3}+\lambda_{3,4}v_{3,4}=0\\
-\lambda_{1,4}v_{1,4}+\lambda_{2,4}v_{2,4}-\lambda_{3,4}v_{3,4}=0.\\
\end{cases}\]
This new system is equivalent to the one determined by Equations (\ref{dets2Rel}) when $d=2$. Since the map $det^{S^2}$ is linear (and so the sign in front of $F_{i,j}$ does not mater),  we conclude that the above equilibrium problem has a nontrivial solution if and only if $det^{S^2}(\otimes_{1\leq i<j\leq 4}(v_{i,j}))=0$, or equivalently  $det^{S^2}(\otimes_{1\leq i<j\leq 4}(F_{i,j}))=0$. 

Next, consider five particles in a two dimensional space. The equilibrium problem will correspond to a system that has $5$ vector equations (or equivalently $10$ linear equations), and $10$ variables $\lambda_{i,j}$ for $1\leq i<j\leq 5$. More precisely we have 
\[\begin{cases}
\lambda_{1,2}F_{1,2}+\lambda_{1,3}F_{1,3}+\lambda_{1,4}F_{1,4}+\lambda_{1,5}F_{1,5}=0\\
\lambda_{2,1}F_{2,1}+\lambda_{2,3}F_{2,3}+\lambda_{2,4}F_{2,4}+\lambda_{2,5}F_{2,5}=0\\
\lambda_{3,1}F_{3,1}+\lambda_{3,2}F_{3,2}+\lambda_{3,4}F_{3,4}+\lambda_{3,5}F_{3,5}=0\\
\lambda_{4,1}F_{4,1}+\lambda_{4,2}F_{4,2}+\lambda_{4,3}F_{4,3}+\lambda_{4,5}F_{4,5}=0\\
\lambda_{5,1}F_{5,1}+\lambda_{5,2}F_{5,2}+\lambda_{5,3}F_{5,3}+\lambda_{5,4}F_{5,4}=0.
\end{cases}\]
However, since $F_{i,j}=-F_{j,i}$ and $\lambda_{i,j}=\lambda_{j,i}$, the $5$ vector equations are linearly dependent (their sum is zero), and so the system has a nontrivial solution.

More generally, we have the following setting.

\begin{definition}  A $q$-particle $2$-equilibrium system is a collection of vectors $F_{i,j}\in V_d$ for each $1\leq i,j\leq q$ (i.e. the force $F_{i,j}$ acts on the particle $i$), that satisfy the third law of motion, that is    
$$F_{i,j}=-F_{j,i},$$ 
for all $1\leq i,j\leq q$. In particular we have that $F_{i,i}=0$ for all $1\leq i\leq q$. 

We say that the $q$-particle $2$-equilibrium problem associated to the above system  has a nontrivial solution if there exist $\lambda_{i,j}=\lambda_{j,i}\in k$ for every $1\leq i<j\leq q$,  not all zero, that satisfy the vector equations ${\mathcal{F}}_m$ given by
\begin{eqnarray}
  \sum_{\substack{i=1\\ i\neq m}}^{q}\lambda_{m,i}F_{m,i}=0,\label{System2d}
\end{eqnarray}
for all $1\leq m\leq q$ (i.e. the sum of all rescaled forces $\lambda_{m,i}F_{m,i}$ that act on the $m$-th particle is zero). 
\end{definition}

We have the following result. 

\begin{theorem} Consider a $q$-particle $2$-equilibrium system determined by $F_{i,j}\in V_d$ for all $1\leq i,j\leq q$. 
\begin{enumerate}
\item If $q>2d$ then the $q$-particle $2$-equilibrium problem  has a nontrivial solution. 
\item If $q=2d$ then the $2d$-particle $2$-equilibrium problem has a nontrivial solution if and only if $det^{S^2}(\otimes_{1\leq i<j\leq 2d}(F_{i,j}))=0$. 
\end{enumerate}
\end{theorem}
\begin{proof} For the first statement, notice that the number of variables is  $\displaystyle{{q \choose 2}=\frac{q(q-1)}{2}}$, while the number of vector equations is $q$ (i.e. $qd$ linear equations). Moreover, just like in the warm-up example,  the vector equations are dependent (their sum  is zero), and so it is enough to consider only the first $q-1$ vector equations. 
Since $q> 2d$  we have that  $$\frac{q(q-1)}{2}> (q-1)d,$$ which means that we have more variables then linear equations,
and so the corresponding system has nontrivial solutions. 

Next assume that $q=2d$. Again, the $2d$ vector equations are dependent (their sum is zero), so our system is equivalent with the system consisting of only the first $2d-1$ equations. For $1\leq i<j\leq 2d$ we denote $$v_{i,j}=(-1)^{i+j-1}F_{i,j}.$$ This gives $F_{i,j}=(-1)^{i+j-1}v_{i,j}$, and $F_{j,i}=(-1)^{i+j}v_{i,j}$ for all $1\leq i<j\leq 2d$. With this notation Equation (\ref{System2d}) becomes 
\begin{eqnarray}
(-1)^{m-1}\left(\sum_{s=1}^{m-1}(-1)^{s-1}\lambda_{s,m}v_{s,m}+\sum_{t=m+1}^{2d}(-1)^{t}\lambda_{m,t}v_{m,t}\right)=0,\label{System2dV}
\end{eqnarray}
for all $1\leq m\leq 2d-1$. 

Finally, notice that this equation is equivalent to Equation (\ref{dets2Rel}), and so the system given by the Equations (\ref{System2d}) has a nontrivial solution if and only if $det^{S^2}(\otimes_{1\leq i<j\leq 2d}(v_{i,j}))=0$, or equivalently $det^{S^2}(\otimes_{1\leq i<j\leq 2d}(F_{i,j}))=0$.
\end{proof}

\begin{remark} It was shown in \cite{dets2} that for every $d\geq 1$ there exists an element $E_d\in V_d^{\otimes {2d \choose 2}}$ such that $det^{S^2}(E_d)\neq 0$. In particular this means that there are examples of  $2d$-particle $2$-equilibrium problems that have only the trivial solution. 
\end{remark}
\begin{remark} As mentioned above, the equilibrium problem studied in this section is based on Newton's third law of motion,  {\it the forces two particle exert on each other are equal and opposite}. According to \cite{gps} this statement is sometimes called  the weak law of action and reaction. Notice that, it does not assume that the forces act in the direction of line joining the two particle, which in  \cite{gps} is called the strong law of action and reaction. 

It was shown in \cite{dets2} that if there exist $p_1,\dots,p_{2d}\in V_d$ such that $v_{i,j}=p_j-p_i$ then $det^{S^2}(\otimes_{1\leq i<j\leq 2d}(v_{i,j}))=0$.  This means that if a $2d$-particles $2$-equilibrium system obeys the strong law of action and reaction then the corresponding $2d$-particle $2$-equilibrium problem has a nontrivial solution. 
\end{remark}


\section{The $3$-equilibrium Problem}
\label{section3}

In this section we introduce the $3$-equilibrium problem and discuss  how it relates to the map $det^{S^3}$.

To the best of our knowledge, in the context of torque forces there is no equivalent statement for the law of action and reaction. To set up the problem, we are going to assume that when three particles interact they behave according to a law similar to the third law of motion. More precisely we  postulate the following rule. 

Consider three particles $p_1$, $p_2$, $p_3$, such that the particle $p_3$ exerts a force $F_{1,2,3}$ on the the oriented axis determined by $p_1$ and $p_2$. We postulate that in such a situation for each  $1\leq i,j,k\leq 3$ the particle $p_k$ exerts a force $F_{i,j,k}$  on the oriented axis determined by $p_i$ and $p_j$ such  that 
$$F_{i,j,k}=F_{j,k,i}=F_{k,i,j}=-F_{j,i,k}=-F_{i,k,j}=-F_{k,j,i}.$$
In particular we have that $F_{i,i,i}=0$, and $F_{i,i,j}=F_{i,j,i}=F_{j,i,i}=0$ for all $1\leq i,j\leq 3$. 

\begin{remark} \label{remarktorque2}
We are not arguing that such an assumption is reasonable from a physical point of view, we only postulate this rule and investigate its consequences.   

To have some intuition about this setting, one can think that $F_{i,j,k}$ acts as a torque-like force around the oriented axis determined by the particles $i$ and $j$. Notice that $F_{i,j,k}$ and $F_{j,i,k}$ act on the same axis but in opposite directions, which inevitably creates a tension in that axis. 

The $3$-equilibrium problem will investigate under what circumstances the forces $F_{i,j,k}$ can be rescaled such that the sum of the (torque-like) forces on each oriented axis is zero  (i.e. there is no tension in any of the axis). 
\end{remark}

\begin{definition}   A $q$-particle $3$-equilibrium system is a collection of vectors $F_{i,j,k}\in V_d$ for each $1\leq i,j,k\leq q$ (i.e. the force exerted by the $k$-th particle  on the oriented axis $(i,j)$), such  that     
$$F_{i,j,k}=F_{j,k,i}=F_{k,i,j}=-F_{j,i,k}=-F_{i,k,j}=-F_{k,j,i}.$$  In particular we have $F_{i,i,i}=0$, and $F_{i,i,j}=F_{i,j,i}=F_{j,i,i}=0$ for all $1\leq i,j\leq q$. 

We say that the $q$-particles $3$-equilibrium problem associated to the above system  has a nontrivial solution if for every $1\leq i<j<k\leq q$ there exist $\lambda_{i,j,k}=\lambda_{j,k,i}=\lambda_{k,i,j}=\lambda_{j,i,k}=\lambda_{k,j,i}=\lambda_{i,k,j}\in k$,  not all zero, that satisfy the vector equations $\mathcal{F}_{m,n}$ given by
\begin{eqnarray}
\sum_{\substack{i=1\\ i\notin \{m,n\}}}^{q}\lambda_{m,n,i}F_{m,n,i}=0,\label{System3d}
\end{eqnarray}
for all $1\leq m<n\leq q$.  
\end{definition}
Notice that Equation (\ref{System3d}) ensures that the sum of all (rescaled) forces $\lambda_{m,n,i}F_{m,n,i}$ that act on the oriented axis $(m,n)$ is zero. 

\begin{theorem} Consider a $q$-particle $3$-equilibrium system determined by $F_{i,j,k}\in V_d$ for each $1\leq i,j,k\leq q$. 
\begin{enumerate}
\item If $q>3d$ then the $q$-particle $3$-equilibrium problem  has a nontrivial solution. 
\item If $q=3d$ then the $q$-particle $3$-equilibrium problem has a nontrivial solution if and only if $det^{S^3}(\otimes_{1\leq i<j<k\leq 3d}F_{i,j,k})=0$. 
\end{enumerate}
\end{theorem}
\begin{proof} For the first statement notice that the number of variables is  $\displaystyle{{q \choose 3}=\frac{q(q-1)(q-2)}{6}}$, while the number of vector equations is $\displaystyle{\frac{q(q-1)}{2}}$ (i.e. $\displaystyle{d\frac{q(q-1)}{2}}$ scalar equations). However, not all of the equations are linearly independent. More precisely we have the identity 
\begin{equation} \sum_{t=1}^{m-1}\mathcal{F}_{s,m}+\sum_{t=m+1}^{q}\mathcal{F}_{m,t}=0,\label{relmr3}
\end{equation}
for all $1\leq m\leq q-1$. This means that it is enough to consider only those equations $\mathcal{F}_{m,n}$ with $1\leq m< n\leq q-1$. 
And so, we have only $\displaystyle{\frac{(q-1)(q-2)}{2}}$ vector equations (corresponding to $\displaystyle{d\frac{(q-1)(q-2)}{2}}$ linear equations). Since $q>3d$ then  $$\frac{q(q-1)(q-2)}{6}> d\frac{(q-1)(q-2)}{2},$$
and so we have more variables than linear equations. This implies that the corresponding system has nontrivial solutions. 

Next assume that $q=3d$. Just like above the $\displaystyle{\frac{3d(3d-1)}{2}}$ vector equations are dependent (see Equation (\ref{relmr3})), and our system is equivalent with a system consisting of $\displaystyle{\frac{(3d-1)(3d-2)}{2}}$ vector equations $\mathcal{F}_{m,n}$ where $1\leq m<n\leq 3d-1$. 

For $1\leq i<j<k\leq 3d$ we denote $$v_{i,j,k}=(-1)^{i+j+k}F_{i,j,k}.$$ This gives that  $F_{i,j,k}=(-1)^{i+j+k}v_{i,j,k}$, $F_{i,k,j}=(-1)^{i+j+k-1}v_{i,j,k}$ and $F_{j,k,i}=(-1)^{i+j+k}v_{i,j,k}$ for all $1\leq i<j<k\leq 3d$. With these notation Equation (\ref{System3d}) becomes 
\begin{eqnarray*}
(-1)^{m+n-1}\left(\sum_{s=1}^{m-1}(-1)^{s-1}\lambda_{s,m,n}v_{s,m,n}+\sum_{t=m+1}^{n-1}(-1)^{t}\lambda_{m,t,n}v_{m,t,n}+\sum_{u=n+1}^{3d}(-1)^{u+1}\lambda_{m,n,u}v_{m,n,u}\right)=0,\label{System3dV}
\end{eqnarray*}
for all $1\leq m<n\leq 3d-1$. This is equivalent to Equation (\ref{MainEq3}), and so the system given by the Equations (\ref{System3d}) has a nontrivial solution if and only if $det^{S^3}(\otimes_{1\leq i<j<k\leq 3d}(v_{i,j,k}))=0$, or equivalently $det^{S^3}(\otimes_{1\leq i<j<k\leq 3d}(F_{i,j,k}))=0$.
\end{proof}

\begin{example}  \label{exampleCP} Let $p_i\in \mathbb{R}^3$ for all $1\leq i\leq 9$. For all $1\leq i,j,k\leq 9$ we define 
\begin{equation}
F_{i,j,k}=(p_j-p_i)\times (p_k-p_i)=p_i\times p_j+p_j\times p_k+p_k\times p_i, 
\end{equation}
where $\times$ is the cross product in $\mathbb{R}^3$. One can check that $F_{i,j,k}$ determines a $9$-particle $3$-equilibrium system in $\mathbb{R}^3$. We will show that $det^{S^3}(\otimes_{1\leq i<j<k\leq 9}(F_{i,j,k}))=0$. 

Indeed, it is a simple linear algebra exercise to see that there exist $\lambda_i\in \mathbb{R}$ for all $1\leq i\leq 9$ such that 
\begin{eqnarray*}
\sum_{i=1}^9\lambda_ip_i&=&0\\
\sum_{i=1}^9\lambda_i&=&0,
\end{eqnarray*} and at least three of the $\lambda_i$ are not zero. For $1\leq i< j<k\leq 9$ take  $$\lambda_{i,j,k}=\lambda_i\lambda_j\lambda_k,$$ 
and notice that not all $\lambda_{i,j,k}$ are zero. 

For $1\leq m<n\leq 3d$ we have 
\begin{eqnarray*}
&&\sum_{\substack{i=1\\ i\notin \{m,n\}}}^{9}\lambda_{m,n,i}F_{m,n,i}=\sum_{\substack{i=1\\ i\notin \{m,n\}}}^{9}\lambda_{m}\lambda_n\lambda_i(p_m\times p_n+p_n\times p_i+p_i\times p_m)\\
&=&\left(\sum_{i=1}^{9}\lambda_{m}\lambda_n\lambda_i(p_m\times p_n)\right)-\lambda_m\lambda_n\lambda_m(p_m\times p_n)-\lambda_m\lambda_n\lambda_n(p_m\times p_n)+\\
&&\left(\sum_{i=1}^{9}\lambda_{m}\lambda_n\lambda_i(p_n\times p_i)\right)-\lambda_m\lambda_n\lambda_m(p_n\times p_m)-\lambda_m\lambda_n\lambda_n(p_n\times p_n)+\\
&&\left(\sum_{i=1}^{9}\lambda_{m}\lambda_n\lambda_i(p_i\times p_m)\right)-\lambda_m\lambda_n\lambda_m(p_m\times p_m)-\lambda_m\lambda_n\lambda_n(p_n\times p_m)\\
&=&\left(\sum_{i=1}^{9}\lambda_i\right)\lambda_{m}\lambda_n(p_m\times p_n)+(\lambda_{m}\lambda_np_n)\times \left(\sum_{i=1}^{9}\lambda_ip_i\right)+
\left(\sum_{i=1}^{9}\lambda_ip_i\right)\times (\lambda_{m}\lambda_np_m)\\
&=&0
\end{eqnarray*}
Which shows that the $9$-particle $3$-equilibrium problem associated to $F_{i,j,k}$ has a nontrivial solution, and so $det^{S^3}(\otimes_{1\leq i<j<k\leq 9}(F_{i,j,k}))=0$. 
\end{example}

\begin{example} \label{exampleEP} More generally, let $V_s$ be a vector space such that $dim(V_s)=s\geq 3$, and consider vectors $v_t\in V_s$ for all $1\leq t\leq 3d=3\displaystyle{{s \choose 2}}$. Take the vector space  $V_s\wedge V_s$  which we identify to $V_d$, where $d=\displaystyle{{s \choose 2}}$  (a basis in $V_d$ is given  by $\{e_i\wedge e_j\;|\; 1\leq i<j\leq s\}$). 
 For each $1\leq i,j,k\leq 3d=3\displaystyle{{s \choose 2}}$ we define: 
$$F_{i,j,k}=v_i\wedge v_j+v_j\wedge v_k+v_k\wedge v_i.$$
It is easy to check that $F_{i,j,k}$ determines a $3d$-particle $3$-equilibrium system in $V_d=V_{{s \choose 2}}$. 

Just like in Example \ref{exampleCP}  we can find $\lambda_i\in k$ for all $1\leq i\leq 3d$ such that 
\begin{eqnarray*}
\sum_{i=1}^{3d}\lambda_iv_i&=&0\\
\sum_{i=1}^{3d}\lambda_i&=&0,
\end{eqnarray*} and at least three of the $\lambda_i$ are not zero. For $1\leq i< j<k\leq 3d$ take  $$\lambda_{i,j,k}=\lambda_i\lambda_j\lambda_k,$$ 
and notice that not all $\lambda_{i,j,k}$ are zero. One can check that 
\begin{eqnarray*}
\sum_{\substack{i=1\\ i\notin \{m,n\}}}^{3d}\lambda_{m,n,i}F_{m,n,i}=0,
\end{eqnarray*}
and so $det^{S^3}(\otimes_{1\leq i<j<k\leq 3d}(F_{i,j,k}))=0$. For $s=3$ we have $d=3$ and using the usual identification of the wedge product and cross product we recover Example \ref{exampleCP}. 
\end{example} 

\begin{remark} If the postulate made at the beginning of this section would hold in actual life, then it is likely that the magnitude of the force $F_{i,j,k}$ from Example \ref{exampleCP} would be  inverse proportional to the area of the region determined by the three particles $p_i, p_j, p_k\in \mathbb{R}^3$. So, a more reasonable expression for the force $F_{i,j,k}$ would be 
\begin{equation}
F_{i,j,k}=\frac{C}{\lVert p_i\times p_j+p_j\times p_k+p_k\times p_i\rVert^2}(p_i\times p_j+p_j\times p_k+p_k\times p_i), 
\end{equation}
where $C$ is a constant that depends on the properties of the particles (i.e. mass, charge, etc), and the type of interaction among the three particles. For the purpose of Example \ref{exampleCP} this change does not make any difference. 
\end{remark}
\begin{remark} It was shown in \cite{ls2} that for every $d\geq 1$ there exist an element $E_d^{(3)}\in V_d^{\otimes {3d \choose 3}}$ such that $det^{S^3}(E_d^{(3)})\neq 0$. In particular this means that there are examples of $3d$-particles $3$-equilibrium problems that have only the trivial solution. 
\end{remark}
\begin{remark} Notice that in this section we consider interactions among triples of particles, but ignore interactions between pairs of particles.  From a physics point of view, this is probably not a good assumption. One might hope to be able to deal at the same time with both types of interactions, between pairs of particles (as in Section \ref{section2}), and among triples of particles (as in this section). At this point it is not clear to us what the correct setting for such a problem is, or how to find the corresponding determinant-like map. 
\end{remark}


\section{The $r$-equilibrium Problem}
\label{sectionr}

In this section we introduce the $q$-particle $r$-equilibrium problem and show how it relates to the map $det^{S^r}$.  We present only the main ideas and leave all of the details to the reader.

\begin{definition}   
A $q$-particle $r$-equilibrium system is a collection of vectors $F_{i_1,i_2,\dots,i_r}\in V_d$ for each  $1\leq i_1,i_2,\dots,i_r\leq q$ (i.e. the force exerted by the $i_r$-th particle on the oriented $(r-2)$-space determined by the particles $i_1,\dots, i_{r-1}$), such that 
$$F_{i_1,\dots,i_r}=sign(\sigma)F_{\sigma(i_1),\dots,\sigma(i_r)}$$ for all $\sigma$ permutations of the set $\{i_1,i_2,\dots,i_r\}$, where $sign(\sigma)$ is the signature of the permutation $\sigma$. In particular we have $F_{i_1,\dots,j,\dots,j,\dots i_r}=0$. 

We say that the $q$-particles $r$-equilibrium problem associated to the above system  has a nontrivial equilibrium solution if there exist $\lambda_{i_1,\dots,i_r}=\lambda_{\sigma(i_1),\dots,\sigma(i_r)}$ for every $1\leq i_1<\dots<i_r\leq q$ and all $\sigma$ permutations of the set $\{i_1,i_2,\dots,i_r\}$,  not all zero, that satisfy the vector equation $\mathcal{F}_{m_1,\dots, m_{r-1}}$ given by
\begin{eqnarray}
\sum_{\substack{i=1\\ i\notin \{m_1,\dots,m_{r-1}\}}}^{q}\lambda_{m_1,\dots,m_{r-1},i}F_{m_1,\dots,m_{r-1},i}=0,\label{SystemRd}
\end{eqnarray}
for all $1\leq m_1<\dots<m_{r-1}\leq q$. 
\end{definition}

\begin{theorem} Consider a $q$-particle $r$-equilibrium system determined by $F_{i_1,\dots,i_r}\in V_d$ for $1\leq i_1,\dots,i_r\leq q$.  
\begin{enumerate}
\item If $q>rd$ then the $q$-particles $r$-equilibrium problem  has a nontrivial solution. 
\item If $q=rd$ then the $q$-particles $r$-equilibrium problem has a nontrivial solution if and only if $det^{S^r}(\otimes_{1\leq i_1<i_2<\dots<i_r\leq rd}F_{i_1,i_2,\dots,i_r})=0$. 
\end{enumerate}
\end{theorem}
\begin{remark} It is not difficult to  generalize Example \ref{exampleEP} to any $r$. However, for  $r\geq 4$ it is not know yet if the map $det^{S^r}$ is nontrivial for all $d$. Moreover, when $r\geq 4$ we do not have any intuitive example of an  $r$-equilibrium system. Because of this we will not elaborate any further for the case $r\geq 4$. 
\end{remark}

\section*{Declarations of interest} 
There is no conflict of interest.




\section*{Acknowledgment}
We thank  Steve Lippold and Alin Stancu for comments on an earlier version of this paper. 

\bibliographystyle{amsalpha}

 \end{document}